\documentclass[11pt,a4paper]{amsart}
\usepackage{amssymb,amsmath,amsthm,amsopn,amsfonts,amscd,enumerate,mathtools}
\usepackage[foot]{amsaddr}
\usepackage{tikz}
\usepackage{tikz-cd}
\usepackage{enumitem}
\usepackage{amsaddr}
\usepackage{etoolbox}
\usepackage{upgreek}\usetikzlibrary{matrix}
\usepackage[utf8]{inputenc}
\usepackage{hyperref}
\usepackage{thmtools,thm-restate}
\newcommand\Pic{\operatorname{Pic}}
\newcommand\Br{\operatorname{Br}}
\newcommand{\proset}{\,\mathrel{\lower 4pt\hbox{$\scriptscriptstyle/$}
\mkern -14mu\subseteq }\,}

\newcounter{thmcount}
\newtheorem{defn}{Definition}[section]

\newtheorem{lem}[defn]{Lemma}

\newtheorem{prop}[defn]{Proposition}
\newtheorem*{prop*}{Proposition}
\newtheorem{thm}[defn]{Theorem}
\newtheorem{cor}[defn]{Corollary}

\newtheorem*{claim*}{Claim}

\theoremstyle{remark}

\theoremstyle{remark}

\theoremstyle{remark}
\newtheorem{exmp}[defn]{Example}
\AtEndEnvironment{exmp}{\hfill$\Diamond$}
\theoremstyle{remark}

\theoremstyle{remark}

\theoremstyle{remark}

\theoremstyle{remark}

\theoremstyle{remark}

\theoremstyle{remark}
\newtheorem{rmk}[defn]{Remark}
\theoremstyle{remark}

\numberwithin{equation}{section}

 \makeatother
\title[Relative Brauer groups and Kummer sequence in $fppf$-topology]{Relative Brauer groups and Kummer Sequence without characteristic constraint in $fppf$-topology}
\author{Sourayan Banerjee}
\address{Department of Mathematics \\ Indian Institute of Technology Kharagpur\\west Bengal-721302, India}
\email{sourayanbanerjee@gmail.com (Corresponding author)}
\date{}
\date{} 
 
\begin{document}
\maketitle
\begin{abstract}
    Let $f: X\rightarrow S$ be a faithful affine map. In this article we construct a group homomorphism $\theta_{{fppf}}: \text{Br}(f) \longrightarrow H^1_{{fppf}}\!\left(S, (f_*\mathcal{O}_X^\times / \mathcal{O}_S^\times)_{{fppf}}\right)$, extending the formerly defined morphism $\theta_{et}$ over an \'etale site. In doing so, we eliminate the characteristic constraint previously found in the relative version of Kummer's exact sequence. Furthermore, given a finite subintegral extension of noetherian rings $A \hookrightarrow B$, we show that all the $fppf$ cohomology groups, $H^i_{fppf}(Spec(A),\mu^f_n); \;\forall i\geq 0$ vanish whenever $n$ is coprime to the positive characteristic of $A$. 
\end{abstract}

\section{Introduction}
For a morphism of schemes $f: X \to S$, the relative Brauer group $\text{Br}(f)$ in the sense of Bass \cite{Bass} is defined to be the Grothendieck group $K_0(\text{Az}(f^*))$ of relative Azumaya triples. Where the classical relative Brauer group $\text{Br}(X/S) := \ker(\text{Br}(S) \to \text{Br}(X))$, the group $\text{Br}(f)$ measures $K$-theoretic invariants of the base $S$ and hence can be placed into a natural exact sequence:$$\text{Pic}(S) \longrightarrow \text{Pic}(X) \longrightarrow \text{Br}(f) \longrightarrow \text{Br}(S) \longrightarrow \text{Br}(X).$$ In \cite{VS}, the author constructed a natural group homomorphism:$$\theta_{\text{et}}: \text{Br}(f) \longrightarrow H^1_{\text{et}}\!\left(S, (f_*\mathcal{O}_X^\times / \mathcal{O}_S^\times)_{\text{et}}\right)$$ for any faithful affine map $f: X \to S$ of noetherian schemes. Furthermore, in \cite[\S~6]{VS} the author established a relative variant of Kummer's exact sequence in the \'etale topology:$$0 \longrightarrow \mu_n^f \longrightarrow (f_*\mathcal{O}_X^\times / \mathcal{O}_S^\times)_{\text{et}} \xrightarrow{\;\;[n]\;\;} (f_*\mathcal{O}_X^\times / \mathcal{O}_S^\times)_{\text{et}} \longrightarrow 0$$which yields an explicit description of the relative $n$-torsion subgroups through the low-degree long exact sequence in \'etale cohomology. While the construction of $\theta_{\text{et}}$ function in \cite{VS} is generally for faithful affine maps, the exactness of the relative Kummer sequence in the \'etale topology requires two critical restrictions \cite[Proposition~6.2]{VS}. Firstly, $f: X \to S$ must be a faithful finite map, and secondly, the integer $n$ must be invertible in the residue fields $k(s)$ for all $s \in S$ (i.e., $\text{char}(k(s)) \nmid n$). Thus, in arbitrary characteristic, the multiplication map $x \mapsto x^n$ on unit sheaves fails to be surjective in the \'etale topology. Extracting $n$-th roots $T^n - u = 0$ over a base ring $R$ yields purely inseparable algebra extensions $R \to R[T]/(T^n - u)$, which are finite and flat, but not \'etale. Consequently, the \'etale relative Kummer sequence breaks down for prime power exponents, leaving the computation of the relative cohomological Brauer group beyond the scope of \'etale methods. One of the main contributions of this article is to overcome this characteristic limitation by translating the framework found in \cite{VS} to the ${fppf}$ topology (faithfully flat and finitely presented). We explicitly show that even though the machinery developed by Sadhu does not lift naturally to the ${fppf}$ site $S_{{fppf}}$ because of non-isomorphic $\check{c}ech$ and $fppf$ cohomology, there is a way around via sheaf torsors. Hence, as an integral part of this article, we construct a natural map in $\S~$\ref{theta} with the help of sheaf torsors (contrast to the presheaf torsors used in \cite{VS}) as follows :
\begin{thm}\label{thetafp}
Let $f: X \to S$ be a faithful affine map of noetherian schemes, then there exists a canonical group homomorphism
$$\theta_{{fppf}}: \text{Br}(f) \longrightarrow H^1_{{fppf}}\!\left(S, (f_*\mathcal{O}_X^\times / \mathcal{O}_S^\times)_{{fppf}}\right).$$
\end{thm}

In $\S~$\ref{kummer}, we also generalize the relative Kummer's sequence mentioned above for arbitrary characteristic. To be precise, what we establish is the following:
\begin{thm}
For a faithfully flat finite map of noetherian schemes $f:X \to S$ and any $n\geq 1$ in $\mathbb{N}$, the sequence 
$$0 \longrightarrow \mu_n^f \longrightarrow (f_*\mathcal{O}_X^\times / \mathcal{O}_S^\times)_{{fppf}} \xrightarrow{\;\;[n]\;\;} (f_*\mathcal{O}_X^\times / \mathcal{O}_S^\times)_{{fppf}} \longrightarrow 0$$is short exact in $S_{{fppf}}$.
\end{thm}

Applying the long exact sequence of $fppf$-cohomology, the following theorem immediately holds for any such faithfully flat finite map $f: X \to S$ between noetherian schemes.
\begin{thm}
The sequence
$$0 \longrightarrow \text{Pic}(f) \otimes \mathbb{Z}/n\mathbb{Z} \longrightarrow H^1_{{fppf}}(S, \mu_n^f) \longrightarrow {}_n H^1_{{fppf}}\!\left(S, (f_*\mathcal{O}_X^\times / \mathcal{O}_S^\times)_{{fppf}}\right) \longrightarrow 0$$ is exact for any nonzero natural number $n$. 
\end{thm}
This completely generalizes \cite[Theorem~6.3]{VS} to arbitrary characteristic when flatness is assumed, classifying relative $\mu_{p^k}^f$-torsors and capturing the $p$-torsion of the relative Brauer group via flat group scheme cohomology for a prime $p$. 

Finally, in $\S~$\ref{subint}, we recall the basics of subintegral extensions and show that over $\mathbb{Q}$-algebras, the results established in \cite[Theorem~5.1]{VS} follow verbatim. We also prove that working over $S_{fppf}$ ensures the following generalizations hold:
\begin{thm}
    Let $f: A \to B$ be a finite subintegral extension of noetherian rings and $A$ be of characteristic $m > 0$. Then the following holds 
   \begin{enumerate}
   \item $H^i_{{fppf}}(\text{Spec}(A), \mu_{n}^{f}) = 0 $ for all $i \ge 0$ and $n$ is invertible in $A$.
   \item If $m = p$, a prime and $n=p^k$ for $k\geq 1$, then $H^i_{{fppf}}(\text{Spec}(A), \mu_{n}^{f})$ does not vanish in general.
   \end{enumerate} 
\end{thm}

\section{Azumaya algebras and Relative Brauer Group}\label{relBr}
 
Throughout this section, we will assume that $S$ is a noetherian scheme and $f: X \to S$ is a faithful, affine morphism of schemes, i.e., affine and the map $\mathcal{O}_S \to f_*\mathcal{O}_X$ is injective. We denote by $S_{{fppf}}$ the small ${fppf}$ site over $S$, consisting of faithfully flat, finitely presented morphisms $U \to S$. Below, we recall some basic definitions related to Azumaya algebras and Brauer groups that will be used in later sections. For a detailed reading regarding Brauer groups, we refer to \cite{Aus},\cite{MC}.
\begin{defn}
An $\mathcal{O}_S$-algebra $A$ is an Azumaya algebra over $S$ if $A$ is locally free of finite rank as an $\mathcal{O}_S$-module and the canonical map$$A \otimes_{\mathcal{O}_S} A^{\mathrm{op}} \longrightarrow \mathrm{End}_{\mathcal{O}_S}(A), \quad a \otimes b^{\mathrm{op}} \mapsto (x \mapsto axb)$$is an isomorphism of $\mathcal{O}_S$-algebras.
\end{defn}
Recall that $A$ and $A'$ as Azumaya algebra over the scheme $S$, are said to be similar if there exists locally free $\mathcal{O}_S$-modules $E$ and $E'$ such that $$A \otimes_{\mathcal{O}_S} \mathrm{End}(E) \cong A' \otimes_{\mathcal{O}_S} \mathrm{End}(E').$$
The collection of the similarity classes of Azumaya algebras, denoted as $[A]$, under the multiplication $[A]\cdot[A'] = [A\otimes A']$ forms an abelian group which is defined to be the Brauer group over $S$; $\mathrm{Br}(S)$.
We also recall the definition of the category $\mathrm{Az}(S)$ before proceeding further.
\begin{defn}\cite[\S~2.1]{VS}
    Let $\mathrm{Az}(S)$ be the category whose objects are Azumaya algebras over $S$, and the set of morphisms between two objects $A$ and $A'$ is defined by $$Hom_{\mathrm{Az}(S)}(A,A'):= \Delta (A,A')/\sim,$$ where $\Delta(A,A')$ is a set consisting of triples $(P,u,Q)$ with $P,Q$ are locally free $\mathcal{O}_S$-modules of finite rank and $u:A \otimes \mathrm{End}(P) \cong A' \otimes \mathrm{End}(Q)$ is an isomorphism of algebras. $(P,u,Q) \sim (P',u',Q')$ if there exist locally free $\mathcal{O}_S$-modules of finite rank $E$ and $E'$ over $S$ such that $P\otimes E \cong P'\otimes E'$ and $Q\otimes E \cong Q'\otimes E'$.
\end{defn}
For a brief discussion on the above and how in a similar fashion $\tilde{\mathrm{Az}}(S)$ is defined with the introduction of $\tilde{\Delta}(A,B)$, we refer to \cite[\S~2]{VS}.

Let us now recall the construction of the category $\mathrm{Az}(f^*)$ as was done in \cite[\S~2.3]{VS}. The objects are triples $(A_1, \alpha, A_2)$, where $A_1, A_2 \in \mathrm{Az}(S)$ are Azumaya algebras over $S$, and $\alpha: f^* A_1 \xrightarrow{\sim} f^* A_2$ is an isomorphism in the category of Azumaya algebras $\mathrm{Az}(X)$. The morphisms in this category are to be taken as follows: A morphism $(u, v): (A_1, \alpha, A_2) \to (A_1', \alpha', A_2')$ is a pair of Azumaya algebra morphisms $u: A_1 \to A_1'$ and $v: A_2 \to A_2'$ over $S$ such that the following diagram in $\mathrm{Az}(X)$ commutes:$$\begin{array}{ccc} f^*A_1 & \xrightarrow{\;\;\alpha\;\;} & f^*A_2 \\ \Big\downarrow\scriptstyle f^*u && \Big\downarrow\scriptstyle f^*v \\ f^*A_1' & \xrightarrow{\;\;\alpha'\;\;} & f^*A_2' \end{array}$$
The definition of the relative Brauer group, as Sadhu defined following Bass \cite{Bass}, we record it below.
\begin{defn}\label{relbr}
The relative Brauer group $\mathrm{Br}(f)$ is defined as the Grothendieck group $K_0(\mathrm{Az}(f^*))$. It is generated by isomorphism classes $[(A_1, \alpha, A_2)]$ subject to the relations:
\begin{enumerate}
\item $[(A_1, \alpha, A_2)] + [(A_1', \alpha', A_2')] = [(A_1 \otimes A_1', \alpha \otimes \alpha', A_2 \otimes A_2')]$.

\item $[(A_1, \alpha, A_2)] + [(A_2, \beta, A_3)] = [(A_1, \beta\alpha, A_3)]$.
\end{enumerate}
\end{defn}
In \cite[\S~2.4]{VS}, the author explicitly constructs the category $\tilde{\mathrm{Az}}(f^*)$ and gives a natural map between two relative Brauer groups, which we now recall as a lemma.
\begin{lem}\label{BrtoBr}
   Let $\tilde{\Br}(f):= K_0(\tilde{\mathrm{Az}}(f^*)$, then there is a natural group homomorphism $$\Upsilon : \Br(f) \rightarrow \tilde{\Br}(f).$$
\end{lem}
\begin{proof}
    For the proof, we refer to \cite[\S~2.4]{VS}.
\end{proof}
\section{Relative Picard groups and Torsors}\label{relpictor}
This section comprises definitions regarding the relative Picard group, $fppf$ sheafification, and torsors. All the relevant definitions can be found in \cite[\S~4]{VS} over \'etale topology. Although one key thing to note in this section is the definition of the sheaf torsors, contrary to the presheaf torsors defined in \cite{VS}. 
\begin{defn}\cite[\S~4.4]{VS}\label{Picf}
  The relative $\mathrm{Pic}(f)$ is the abelian group generated by $[L_1, \alpha,L_2]$, where the $L_i$ are line bundles on $S$ and $\alpha: f^*L_1 \rightarrow f^*L_2$ is an isomorphism. The relations are:
  \begin{enumerate}
      \item $[L_1,\alpha,L_2] + [L_1',\alpha',L_2'] = [L_1 \otimes L_1', \alpha \otimes \alpha', L_2 \otimes L_2']$;
      \item $[L_1,\alpha,L_2] + [L_2,\beta,L_3] = [L_1, \beta\alpha, L_3]$;
      \item $[L_1, \alpha, L_2]=0$ if $\alpha = f^*(\alpha_0)$ for some $\alpha_0 : L_1 \cong L_2$.
  \end{enumerate}
\end{defn}
For more details and relevant properties, please refer to \cite[\S~4]{VS}.

Below, we define the relative Picard presheaf and its ${fppf}$ sheafification, where for any scheme $U \in S_{{fppf}}$,  $f_U: X \times_S U \to U$ will denote the base change map. 
\begin{defn}\label{presheafpicf}
The relative Picard presheaf $\mathrm{Pic}^f$ on $S_{{fppf}}$ assigns to each flat cover $U \to S$ the group $\mathrm{Pic}(f_U)$.
\end{defn}

Similarly, as it was defined in \cite[\S~4.5]{VS}, we introduce the  relative $\tilde{\Pic}(f)$ below.
\begin{defn}
    The abelian group generated by $[L_1, \alpha, L_2]$, where the $L_i$'s are line bundles on $S$, $\alpha : f^*L_1 \otimes f^*L_2^{-1} \rightarrow f^*L_1^{-1} \otimes f^*L_2$ is an isomorphism and relations are similar to $\Pic(f)$ is defined to be $\tilde{\Pic}(f)$.
\end{defn}
In accordance to Definition \ref{presheafpicf} we also define the preseheaf $\tilde{\Pic}^f$ to be the $fppf$-presheaf on $S_{fppf}$ such that $\tilde{\Pic}^f(U) = \tilde{\Pic}(f_U)$. We also recall the natural group homomorphism (see \cite[\S~4.5]{VS}) $$\psi: \tilde{\Pic}(f) \rightarrow \Pic(f),$$ that maps $[L_1, \alpha, L_2]$ to $[L_1 \otimes L_2^{-1}, \alpha, L_1^{-1}\otimes L_2]$.

\begin{prop}\label{picf}
 Let $f: X \to S$ be a faithful affine map. Let $\mathcal{I}_{{fppf}} := (f_*\mathcal{O}_X^\times / \mathcal{O}_S^\times)_{{fppf}}$ denote the ${fppf}$-sheafification of the relative unit quotient presheaf, and $\mathcal{P}ic^f_{{fppf}}$ the fppf sheafification of the presheaf $\mathrm{Pic}^f$. Then there is a natural isomorphism of ${fppf}$-sheaves:
$$\mathcal{P}ic^f_{{fppf}} \cong \mathcal{I}_{{fppf}}$$
 \end{prop}
 \begin{proof} For any $U \in S_{{fppf}}$, the fundamental exact sequence of relative Picard groups (see \cite[equation~4.3]{VS}) yields the presheaf sequence 
 \begin{equation}\label{sespic}
1 \longrightarrow \mathcal{O}_S^\times(U) \longrightarrow f_*\mathcal{O}_X^\times(U) \longrightarrow \mathrm{Pic}(f_U) \longrightarrow \mathrm{Pic}(U) \longrightarrow \mathrm{Pic}(X_U).
 \end{equation}
 
 In the $S_{{fppf}}$ site, every line bundle $L \in \mathrm{Pic}(U)$ is locally trivial. Hence, the associated $fppf$ sheaf to the presheaf $U \mapsto \mathrm{Pic}(U)$ is the zero sheaf. Thus, if we apply the exact ${fppf}$-sheafification functor $(\cdot)_{{fppf}}$ to the presheaf sequence it recovers the canonical isomorphism $$\mathcal{P}ic^f_{{fppf}} \;\cong\; (f_*\mathcal{O}_X^\times / \mathcal{O}_S^\times)_{{fppf}} =: \mathcal{I}_{{fppf}}.$$ This completes the proof. 
 \end{proof}
 Let $U \to S$ be any ${fppf}$ cover and let $[L_1, \alpha, L_2] \in \tilde{\Pic}^f(U) = \tilde{\Pic}(f_U)$. Since $L_1$ and $L_2$ are line bundles on $U$, there exists an ${fppf}$ cover $\{V_i \to U\}_{i \in I}$ such that $L_1\vert{}_{V_i} \cong L_2\vert{}_{V_i} \cong \mathcal{O}_{V_i}$ for all $i \in I$. If we restrict the class $[L_1, \alpha, L_2]$ to $V_i$, the line bundles become trivial. Therefore, $f_{V_i}^* L_1 \otimes f_{V_i}^* L_2^{-1} \cong \mathcal{O}_{X \times_S V_i}$, which reduces the  the twisted isomorphism $\alpha\vert{}_{V_i}$ to an automorphism of $\mathcal{O}_{X \times_S V_i}$.

 With the above observation, we record the following proposition that connects  $\Pic^f, \tilde{\Pic}^f$ and $\mathcal{P}ic^f_{fppf}$.
\begin{prop}\label{allpic}
    Let $f: X \rightarrow S$ be a faithful affine map and let the $fppf$-sheafification of $\tilde{\Pic}^f$ be denoted as $\tilde{{Pic}}_{fppf}^f$. Then $\mathcal{P}ic^f_{fppf} \cong \tilde{{Pic}}^f_{fppf}$. 
    \end{prop}
\begin{proof}

We are required to show that $\psi_{fppf} : \tilde{{Pic}}^f_{fppf} \to \mathcal{P}ic^f_{fppf}$, is an isomorphism of $fppf$-sheaves. Thus, it suffices to show local injectivity and surjectivity of $\psi_{fppf}$. 

To prove injectivity, it suffices to show that $x$ is locally zero where $U \in S_{fppf}$ and $x \in \tilde{{Pic}}^f_{fppf}(U)$ such that $\psi_{fppf}(x) = 0$.  Let us choose an $fppf$ cover $V \to U$, such that $x|_V$ is represented by $x_V=[L_1, \alpha, L_2] \in \tilde{\Pic}(f_V)$. Due to $L_1$ and $L_2$ being line bundles we immediately have $L_1 \cong L_2 \cong \mathcal{O}_V$. Thus, $x_V = [\mathcal{O}_V, u, \mathcal{O}_V]$ for some $u \in \Gamma(X_V,\mathcal{O}^{\times}_{X_V})$. Now, applying $\psi$ on $x_V$ we get $\psi(x_V) = [\mathcal{O}_V, u ,\mathcal{O}_V] \in \Pic(f_V)$. Since $\psi_{fppf}(x) = 0$ there exists an $fppf$ cover $W \to V$ such that $\psi(x_W) =0$ in the presheaf group $\Pic(f_W)$. Now we apply \cite[Equation~4.3]{VS} to the morphism $f_W: X_W \to W$, which ensures the exactness of the following sequence:
$$ 1 \to \Gamma(W, \mathcal{O}^{\times}_W) \to \Gamma(X_W, \mathcal{O}^{\times}_{X_W}) \to \Pic(f_W).$$ Triviality of $\psi(x_W)$ along with the above exact sequence concludes $u = f^*_W(v)$, for some $v \in \Gamma(W, \mathcal{O}^{\times}_W).$ Now, we already have $x_W = [\mathcal{O}_W, u, \mathcal{O}_W] \in \tilde{\Pic}(f_W)$, where the defining morphism $u : \mathcal{O}_{X_W} \rightarrow \mathcal{O}_{X_W}$, and $u = f^*_W(v)$, therefore the defining isomorphism is the pullback of $v: \mathcal{O}_W \xrightarrow{\sim} \mathcal{O}_W$. Thus, we obtain $[\mathcal{O}_W, u, \mathcal{O}_W]=0$ from the definition of $\tilde{\Pic}(f_W)$. Hence $x$ is locally zero in the $fppf$ topology. Therefore $\ker(\psi_{fppf}) = 0$. Concluding the injectivity.

Now, we proceed to prove local surjectivity. Let $y = [M_1,\beta,M_2] \in \Pic(f_U)$. For and $fppf$ cover $V' \to U$ we get $y|_{V'} = [\mathcal{O}_{V'}, u',\mathcal{O}_{V'}]$ for some $u' \in \mathcal{O}^{\times}_{X_V}$. But if we consider $y' = [\mathcal{O}_{V'}, u',\mathcal{O}_{V'}] \in \tilde{\Pic}(f_{V'})$, then by definition $\psi_{V'}(y') = y|_{V'}$. Hence proving local surjectivity of $\psi_{fppf}.$

\end{proof}
Before proceeding, we refer to \cite[\S~4.3]{VS} for the definition of presheaf torsors, which follows verbatim in any $fppf$ site. Below, we record the definition of sheaf torsors in $S_{{fppf}}$. 
\begin{defn}\cite{Milne}
Let $G$ be an abelian sheaf (sheaf of abelian groups) on the site $S_{{fppf}}$. A $G$-torsor on $S_{{fppf}}$ is a sheaf of sets $\mathcal{F}$ on $S_{{fppf}}$ equipped with a left action $\rho: G \times \mathcal{F} \to \mathcal{F}$ such that:
\begin{enumerate}
\item If $U \in S_{{fppf}}$, and $\mathcal{F}(U) \neq \emptyset$, then the action $G(U) \times \mathcal{F}(U) \to \mathcal{F}(U)$ is simply transitive;

\item for every $U \in S_{fppf}$ there exists an ${fppf}$ cover $\{U_i \to U\}_{i \in I}$ such that $\mathcal{F}(U_i) \neq \emptyset$ for all $i \in I$.
\end{enumerate}
\end{defn}
We will denote $\mathrm{Tor}_{fppf}(S,G)$ as the collection of $G$-torsors. Also, please note that even though torsors can be defined over non-abelian sheaves, in our context, we adhere to sheaves of abelian groups. For more results regarding torsors, we refer to \cite{Milne}, \cite[\S~4.3]{VS}.
\begin{thm}\cite{Milne}\label{isoG}
 Isomorphism classes of $G$-torsors over $S_{{fppf}}$, $\left(\mathrm{Tor}_{fppf}(S,G)/\simeq\right)$ are in canonical $1$-to-$1$ correspondence with elements of the first ${fppf}$ cohomology group $H^1_{{fppf}}(S, G)$ and since $G$ is an abelian sheaf, $\left(\mathrm{Tor}_{fppf}(S,G)/\simeq \right) \cong H^1(S,G)$ as an abelian group.
 \end{thm}
 \begin{proof} This is the standard classification theorem for torsors on Grothendieck sites (see \cite[Chapter III, Theorem 4.3]{Milne}).
 \end{proof}

\section {Construction of the Homomorphism $\theta_{{fppf}}$}\label{theta}

To briefly revisit the construction of the map $\theta_{\text{et}}$ in \cite[\S~ 4]{VS} we find that the author first sends an element $\mathcal A:= [(A_1, \alpha, A_2)] \in \tilde{\Br}(f)$ to the class $[G_{\mathcal{A}}]$ of a presheaf torsor over the site $S_{\text{et}}$. Then use the isomorphism of \'{e}tale and $\check{cech}$ cohomology to establish the rest of the proof. But at the  core of the proof there is a significant observation which we briefly state here for recalling. Since over $S_{\text{et}}$, an Azumaya algebra is locally isomorphic to the endomorphism algebra of a non-trivial locally free module, he fixes $A_i\vert{}_U \cong \text{End}(E_i)$ for $i=1,2$ (see \cite[\S~4.1]{VS}). Now, when matching two local trivializations $(E_1, \tau_1)$ and $(E_2, \tau_2)$, the isomorphism $\alpha : f^*A_1 \cong f^*A_2$ produces a pairing of the form \cite[\S~4.2]{VS}: $$\nu : f_U^*E_1 \otimes (f_U^*E_2)^\vee \xrightarrow{\sim} (f_U^*E_1)^\vee \otimes f_U^*E_2 .$$ One key thing to notice is the appearance of dual sheaves $(\dots)^\vee$. Under a local change of vector bundles $(E_1 \otimes L_1, E_2 \otimes L_2)$, this isomorphism picks up an action of $f_U^*L_1 \otimes f_U^*L_2^{-1}$ on the left and $f_U^*L_1^{-1} \otimes f_U^*L_2$ on the right. To turn this into a well-defined action by the standard relative Picard group $\text{Pic}(f)$, Sadhu introduces the auxiliary group $\tilde{\text{Pic}}({f})$ and its presheaf analogue $\tilde{\text{Pic}}^{f}$ in \cite[\S~4.5]{VS}, which tracks isomorphisms $f^*L_1 \otimes f^*L_2^{-1} \to f^*L_1^{-1} \otimes f^*L_2$, along with a natural homomorphism $\Psi: \tilde{\text{Pic}}({f}) \to \text{Pic}(f)$ \cite[\S~4.5]{VS} to correct for the dualities. 
\begin{rmk}
Shifting the above theory over $S_{\text{fppf}}$ does not follow verbatim, even if an \'{e}tale cover is always an $fppf$ cover by definition. It would be of importance to point out that merely mimicking the approach taken to prove \cite[Theorem~4.12]{VS} will not be feasible as $(f_*\mathcal{O}^{\times}_X/\mathcal{O}^{\times}_S)$ need not be smooth, resulting in non-isomorphic $\check{c}ech$ and $fppf$ first cohomology. 
\end{rmk}
Now, we proceed to construct the natural group homomorphism $$\theta_{{fppf}}: \mathrm{Br}(f) \longrightarrow H^1_{{fppf}}\!\left(S, (f_*\mathcal{O}_X^\times / \mathcal{O}_S^\times)_{{fppf}}\right)$$for any faithful affine map $f: X \to S$ of noetherian schemes by completely bypassing the $\check{c}ech-fppf$ cohomological computation and directly introducing sheaf torsors. In short, what we would like to show is the following: 

By Proposition \ref{picf}, there is a canonical isomorphism of ${fppf}$-sheaves $\mathcal{P}ic^f_{{fppf}} \cong (f_*\mathcal{O}_X^\times / \mathcal{O}_S^\times)_{{fppf}}$. Moreover, by Theorem \ref{isoG}, elements of $H^1_{{fppf}}(S, \mathcal{I}_{{fppf}})$ are in one-to-one correspondence with isomorphism classes of $\mathcal{P}ic^f_{{fppf}}$-torsors over $S_{{fppf}}$. Therefore, the construction of $\theta_{{fppf}}$ reduces to associating an isomorphism class of  $\mathcal{P}ic^f_{{fppf}}$-torsor to a relative Azumaya triple, and establishing that an assignment as such is a canonical group homomorphism. 

\begin{lem}\label{trivial}
Over an fppf cover $U \to S$, $A_1\vert{}_U$ and $A_2\vert{}_U$ become isomorphic to standard matrix algebras over free modules, i.e. $$A_i\vert{}_U \cong \text{End}(\mathcal{O}_U^{\oplus n_i}) \cong \mathcal{M}_{n_i}(\mathcal{O}_U).$$ 
\end{lem}

Below we briefly recall $\mathcal{F}_A$ and  $G_A$ for $fppf$-morphisms as in \cite[\S~4.2]{VS}.
\begin{defn}
For an Azumaya algebra $A \in \mathrm{Az}(S)$, let $\mathcal{F}_A$ denote the fiber category over $S_{{fppf}}$ whose objects are pairs $(E, \tau)$, where $E$ is a locally free $\mathcal{O}_U$-module of finite rank and $\tau: \mathrm{End}_{\mathcal{O}_U}(E) \xrightarrow{\sim} A\vert{}_U$ is an isomorphism of $\mathcal{O}_U$-algebras. Morphisms $g: (E, \tau) \to (E', \tau')$ are isomorphisms of $\mathcal{O}_U$-modules $g: E \xrightarrow{\sim} E'$ such that $\tau'(g \cdot g^{-1}) = \tau$.
\end{defn}

\begin{defn}
Let $A = (A_1, \alpha, A_2) \in \tilde{\mathrm{Az}}(f^*)$ (or in $\mathrm{Az}(f^*))$ be a relative Azumaya triple, where $\alpha: f^*A_1 \xrightarrow{\sim} f^*A_2$ is an isomorphism in $\mathrm{Az}(X)$ then the relative splitting category $G_A$ over $S_{{fppf}}$ associated to $A = (A_1, \alpha, A_2)$ is defined as follows: For every $fppf$ map $U \to S$, objects of $G_A(U)$ are triples $\left((E_1, \tau_1), \nu, (E_2, \tau_2)\right)$ where $(E_1, \tau_1) \in \mathcal{F}_{A_1}(U)$ and $(E_2, \tau_2) \in \mathcal{F}_{A_2}(U)$; $\nu: f_U^* E_1 \otimes (f_U^* E_2)^\vee \xrightarrow{\sim} (f_U^* E_1)^\vee \otimes f_U^* E_2$ is an isomorphism of $\mathcal{O}_{X\times_S U}$-modules. Morphisms are defined to be pairs of vector bundle isomorphisms $(g_1, g_2)$, where $g_1: E_1 \xrightarrow{\sim} E_1'$ and $g_2: E_2 \xrightarrow{\sim} E_2'$ commute with $(\tau_1, \tau_1')$ and $(\tau_2, \tau_2')$ respectively, and satisfy $(f_U^*g_1^\vee \otimes f_U^*g_2) \circ \nu = \nu' \circ (f_U^*g_1 \otimes f_U^*g_2^\vee)$. 
\end{defn}
For a detailed discussion on $\mathcal{F}_A$ and $G_A$ we refer to \cite[\S~4.1,4.2]{VS}. We also note the following proposition before we give an explicit proof of our desired theorem and stick to the notation $[G_A(U)]$ that denotes the set of all isomorphism classes of objects of $G_A(U)$. Hence, assigning $U \mapsto [G_A(U)]$ remains a presheaf on $S_{fppf}$ denoted by $[G_A]$, where the restriction maps are the pullbacks. 
\begin{prop}\label{pretosheaftor}
    Let $A:= (A_1,\alpha, A_2) \in \tilde{\mathrm{Az}}(f^*)$ (or in $\mathrm{Az}(f^*)$) and $S$ is connected. Then $[G_A]$ is a $\tilde{Pic}^f$-torsor. Also, if we denote $[G_A]^+$ to be the $fppf$ sheafification of $[G_A]$, then $[G_A]^+$ is a $\mathcal{P}ic^f_{fppf}$-torsor. For a general noetherian scheme $S$, the argument follows verbatim componentwise.
\end{prop}
\begin{proof}
    First implication is \cite[Proposition~4.3]{VS}. The latter can be canonically established.
     Since sheafification is exact for abelian presheaves and it preserves local surjectivity, the action that is simply transitive and locally non-empty, which defines the presheaf torsor descend to a simply transitive action of the associated sheaf. Hence  $[G_A]^+$ as a $\tilde{Pic}^f_{fppf}$-torsor. Finally, applying Proposition \ref{allpic}, we conclude.
\end{proof}

Now, we proceed to prove Theorem \ref{thetafp}

\begin{proof} Recall that $\mathrm{\tilde{Br}}(f):= K_0(\mathrm{\tilde{Az}}(f^*))$ and let $A:= (A_1, \alpha, A_2) \in \mathrm{\tilde{Az}}(f^*))$. Then the combination of Proposition \ref{picf}, \ref{allpic},  \ref{pretosheaftor}, and Theorem \ref{isoG}, establishes that the isomorphism class of $[G_A]^+$ denoted as $[\mathcal{G}_A] \in H^1_{fppf}(S,f_*\mathcal{O}^{\times}_X/\mathcal{O}^{\times}_S)$. Using \cite[Lemma 4.5]{VS} it's clear that for an isomorphism $\phi: A \to B$ in $\mathrm{\tilde{Az}}(f^*)$ we readily have $[\mathcal{G}_A] = [\mathcal{G}_B]$. Now, \cite[Lemma 4.7(1)]{VS} that produces $[G_{A\otimes A'}] \cong [G_A]\Pi^{\tilde{\Pic}^f}[G_{A'}]$ as presheaf torsors holds true after sheafification as well, i.e.  $[\mathcal{G}_{A\otimes A'}] \cong [\mathcal{G}_A]\Pi^{\tilde{\Pic}^f}[\mathcal{G}_{A'}]$ as $\mathcal{P}ic^f_{fppf}$-torsors. Since $\mathcal{P}ic^f_{fppf}$ is an abelian sheaf we automatically get $[\mathcal{G}_{A\otimes A'}] = [\mathcal{G}_A]+[\mathcal{G}_{A'}]$ in $H^1_{fppf}(S,f_*\mathcal{O}^{\times}_X/\mathcal{O}^{\times}_S)$, as the contracted product corresponds to the addition of their cohomology classes. Now, in a similar fashion from \cite[Lemma 4.7(2)]{VS} we deduce that $[\mathcal{G}_{B\circ A}] = [\mathcal{G}_A]+[\mathcal{G}_B]$. Hence, by assigning $A \mapsto [\mathcal{G}_A]$ we get a well defined group homomorphism $$\omega_{fppf}: \tilde{\Br}(f) \rightarrow H^1_{fppf}(S,f_*\mathcal{O}^{\times}_X/\mathcal{O}^{\times}_S).$$ Finally composing $\omega_{fppf}$ with $\Upsilon$ (see Lemma \ref{BrtoBr}) we get our desired group homomorphism $$\theta_{fppf}: \Br(f) \rightarrow  H^1_{fppf}(S,f_*\mathcal{O}^{\times}_X/\mathcal{O}^{\times}_S).$$

\end{proof}

\section{Kummer Sequence in fppf topology}\label{kummer}
 In \cite[\S~6]{VS}, the author gave an explicit description of a relative Kummer exact sequence for faithful finite maps $f: X \to S$. In the classical \'etale topology, the sequence takes the form (see \cite[Proposition~6.2]{VS} $$0 \longrightarrow \mu_{n}^{f} \longrightarrow \mathcal{I}_{\text{et}} \xrightarrow{\ [n] \ } \mathcal{I}_{\text{et}} \longrightarrow 0$$ where $\mathcal{I}_{\text{et}} := (f_* \mathcal{O}_X^\times / \mathcal{O}_S^\times)_{\text{et}}$ is the relative units sheaf and $\mu_{n}^{f}$ is defined as the kernel sheaf of the $n$-th power multiplication map $[n]: \mathcal{I}_{\text{et}} \to \mathcal{I}_{\text{et}}$. The requirement that $n$ is not divisible by the $\text{char}(k(s)))$ for any $s \in S$ is a direct consequence of working within the \'etale site $S_{\text{et}}$. For the multiplication map $[n]$ to be an epimorphism of sheaves, every section of $\mathcal{I}_{\text{et}}$ over an open set $U \subseteq S$ must locally admit an $n$-th root in the Grothendieck topology. Adjoining an $n$-th root of a local unit $u \in \mathcal{O}^\times(X_U)$ requires passing to the algebra extension $R' = \mathcal{O}(X_U)[T]/(T^n - u)$. For $\text{Spec}(R') \to \text{Spec}(\mathcal{O}(X_U))$ to be a valid cover in $S_{\text{et}}$, the extension must be unramified. Since the derivative of $P(T) = T^n - u$ is $P'(T) = n T^{n-1}$, the extension is \'etale if and only if $P'(T)$ is invertible in $R'$, which necessitates $n \in \mathcal{O}(X_U)^\times$. If $\text{char}(k(s))$ divides $n$, the map $[n]$ ceases to be unramified, breaking the local surjectivity of $[n]$ in $S_{\text{et}}$. This characteristic constraint is not an inherent limitation of relative units, but rather a consequence of the \'etale site's prohibition of ramified and infinitesimal structures.
 
 Let $f: X \to S$ be a faithfully flat and finite map of schemes and $\mu_{n}^{f} := \ker([n]: \mathcal{I}_{\text{fppf}} \to \mathcal{I}_{\text{fppf}})$. Then we recover a characteristic free statement similar to  \cite[Proposition~6.2]{VS}, which we record below.
  
  \begin{thm}\label{relkum}
  For any faithfully flat finite map $f: X \to S$, and any integer $n \ge 1$, the sequence of fppf sheaves:$$0 \longrightarrow \mu_{n}^{f} \longrightarrow \mathcal{I}_{\text{fppf}} \xrightarrow{\ [n]\ } \mathcal{I}_{\text{fppf}} \longrightarrow 0$$is short exact in $S_{\text{fppf}}$, holding universally without any restriction on the characteristic of $k(S)$.
  \end{thm}
  \begin{proof} Exactness at $\mu_{n}^{f}$ and at the middle term $\mathcal{I}_{\text{fppf}}$ holds by the definition of $\mu_{n}^{f}$. To establish the short exact sequence, it remains to show that the map $[n]$ is a local epimorphism of fppf sheaves. It suffices to prove local surjectivity in an affine $U := \mathrm{Spec}(R) \in S_{{fppf}}$. Now, let $\bar{u} \in \mathcal{I}_{\text{fppf}}(U)$ be represented locally by a unit $u \in \mathcal{O}^\times(X_U)$. Now, consider the relative algebra extension $R' = \mathcal{O}(X_U)[T]/(T^n - u)$. Since $u$ is a unit, $R'$ is defined over $\mathcal{O}(X_U)$ by a single polynomial relation, hence making it finitely presented. Furthermore, $R'$ is a finitely generated free $\mathcal{O}(X_U)$-module of rank $n$. Hence it yields a scheme $U' = \text{Spec}(R')$ that represents a valid fppf covering of $U$. Over $U'$, the unit $u$ acquires a local $n$-th root $T$ given by the identity $T^n = u$. Consequently, $[n]$ is surjective as a map of $fppf$ sheaves across all characteristics, including cases where $\text{char}(k(s)) \mid n$ or in mixed characteristic. 
  \end{proof}
  From Theorem \ref{relkum} it is clear that applying the global section functor and taking $fppf$ cohomology $H^i_{{fppf}}(S, -)$ on the relative Kummer's sequence yields the long exact sequence:

$$
\begin{aligned}
0 &\rightarrow H^0_{\mathrm{fppf}}(S,\mu_n^f)
\rightarrow H^0_{\mathrm{fppf}}(S,\mathcal{I}_{\mathrm{fppf}})
\xrightarrow{[n]} H^0_{\mathrm{fppf}}(S,\mathcal{I}_{\mathrm{fppf}})
\rightarrow H^1_{\mathrm{fppf}}(S,\mu_n^f)\\
&\rightarrow H^1_{\mathrm{fppf}}(S,\mathcal{I}_{\mathrm{fppf}})
\xrightarrow{[n]} H^1_{\mathrm{fppf}}(S,\mathcal{I}_{\mathrm{fppf}}).
\end{aligned}
$$
Similar to the proof of \cite[Theorem~6.3, (1)]{VS}, we also quote \cite[Lemma 5.4]{SW}, to establish $H^0_{\text{fppf}}(S, \mathcal{I}_{fppf}) \cong \text{Pic}(f)$. Thus, inducing the fundamental relative Kummer's short exact sequence for $fppf$ topology.
\begin{thm}
     Let $f: X \to S$ be a faithfully flat finite map of schemes and $n$ be a positive integer. Then there is a canonical exact sequence 
 \begin{equation}\label{fund}
0 \longrightarrow \text{Pic}(f) \otimes_{\mathbb{Z}} \mathbb{Z}/n\mathbb{Z} \longrightarrow H^1_{\text{fppf}}(S, \mu_{n}^{f}) \longrightarrow {}_n H^1_{{fppf}}(S, \mathcal{I}_{{fppf}}) \longrightarrow 0
\end{equation}
where ${}_n H^1_{{fppf}}(S, I_{{fppf}})$ denotes the $n$-torsion subgroup.
\end{thm}

\begin{rmk}
The transition from $S_{et}$ to $S_{{fppf}}$ ensures that Kummer's exact sequence and its cohomological applications hold in complete generality, which extends Sadhu’s relative framework done in \cite[\S~6]{VS} to positive-characteristics, $p$-torsion, infinitesimal group schemes, and mixed-characteristic schemes.
\end{rmk}

  \begin{cor}\label{pktor}
  Let $f: X \to S$ be a faithfully flat finite map of schemes and let $p^k$ be a prime power. Then the following sequence, $$0 \longrightarrow \text{Pic}(f) \otimes_{\mathbb{Z}} \mathbb{Z}/p^k\mathbb{Z} \longrightarrow H^1_{{fppf}}(S, \mu_{p^k}^{f}) \longrightarrow {_{p^k}\Br'(f)} \longrightarrow 0$$ is exact, where ${}_{p^k}\Br'(f) := {}_{p^k}H^1_{fppf}\left(S, f_*\mathcal{O}_X^\times / \mathcal{O}_S^\times\right)$ denotes the $p^k$-torsion subgroup of the relative fppf cohomological Brauer group $\Br'(f)$. 
  \end{cor}
  \section{Subintegral Extensions in non-zero characteristic}\label{subint}

  In this section, we work with the subintegral extension of noetherian rings which are not necessarily $\mathbb{Q}$-algebras, thus we deal with nonzero characteristics. Even though \cite[Theorem~5.1]{VS} can not be fully recovered for rings with positive characteristics, we show that \cite[Theorem~6.4]{VS} can be extended for rings that are of nonzero characteristics, provided that we impose a condition.

  A subintegral extension is an integral extension $A \hookrightarrow B$ such that it induces a bijection between $\mathrm{Spec}(B) \to \mathrm{Spec}(A)$ and the residue field extensions are isomorphic at every point. In \cite[Theorem~5.1]{VS} it is recorded that for a noetherian subintegral extension of $\mathbb{Q}$-algebras $A \xhookrightarrow{f} B$, the natural map $f^*: \Br(A) \to \Br(B)$ is an isomorphism and $\Br(f)=0$. If we assume that $A \hookrightarrow B$ is a subintegral extension of $\mathbb{Q}$-algebras, then the theorem holds verbatim in $fppf$ topology as well. The following proposition is a verbatim extension of \cite[Propositon~5.4(3)]{VS2} which essentially establishes \cite[Theorem~5.1]{VS} in $fppf$-topology.
  \begin{prop}\label{isofp}
      Let $f: X \to S$ be a subintegral morphism of Noetherian $\mathbb{Q}$-schemes, then if $S$ is affine and $f$ is finite then for $i>1$, $H^i_{fppf}(S,\mathcal{O}^{\times}_S) \cong H^i_{fppf}(X,\mathcal{O}^{\times}_X)$.
  \end{prop}
  \begin{proof}
      The proof follows the proof of \cite[Propositon~5.4(3)]{VS2} verbatim.
  \end{proof}
  \begin{thm}
      Let $f: A \hookrightarrow B$ be a subintegral extension of noetherian $\mathbb{Q}$-algebras. Then the following holds:
      \begin{enumerate}
          \item $f^* : \Br(A) \to \Br(B)$ is an isomorphism.
          \item $\Br(f) = 0$
      \end{enumerate}
  \end{thm}
  \begin{proof}
      (1) Following the proof of \cite[Propositon~5.4(3)]{VS2}, we also recall that, since $f$ is subintegral $f = \cup_{\lambda}f_\lambda$ where $f_{\lambda} : A \rightarrow B_{\lambda}$ is finite and $B =\cup_{\lambda}B_\lambda$ (see \cite{Swan}). We also establish the same two exact sequences for each $\lambda$, from where it becomes evident that showing $\ker(f^*_{\lambda}) = 0 =\mathrm{coker}(f^*_{\lambda})$ would be sufficient for each $\lambda$.  
      
      Finally, using Proposition \ref{isofp} and a theorem of Gabbar (see \cite{Jong}) we establish the assertion.
      (2) Proof follows verbatim.
  \end{proof}
  In \cite[\S~6]{VS}, it is explained in detail that to establish a relative Kummer's sequence in \'{e}tale topology, it is essential to impose a condition on the characteristic as mentioned in detail in the previous section. The significance of shifting the framework into $fppf$ topology was to be able to establish the relative Kummer's sequence in arbitrary characteristic in Theorem \ref{relkum}. But since not every finite subintegral extension is flat, we can not deduce the result regarding subintegral extensions similar to \cite{VS} directly from the previous section. Thus, we proceed by first revisiting a fundamental example that directly contributes to our final theorem.
 \begin{exmp}\label{nonzero}
    Let $k$ be a field of characteristic $p>0$. Then $k \hookrightarrow \frac{k[x]}{(x^2)}$ is a finite subintegral extension. Now, for any $k$-algebra $R$, it is immediate that $\left(\frac{R[x]}{(x^2)}\right)^{\times} \cong R^{\times}(1+xR)$. After taking the quotient by $R^\times$ on both sides, we get $\mathcal{I}_{fppf}(R) \cong 1+xR$. Therefore, as sheaves of abelian groups we get $\mathcal{I}_{fppf} \cong \mathbb{G}_a$ for every $a \in k$. Since $(1+xa)^{p^k} = 1$ for every $a$, we clearly have $\mu^f_{p^k} \cong \mathbb{G}_a$, and thus we have the following:
    $$H^0_{fppf}(\mathrm{Spec}(k), \mu^f_{p^k}) \cong k \neq 0.$$
 \end{exmp}
  Now, following the proof for \cite[Theorem~6.4 (1)]{VS}, we can get back the same result in $fppf$ topology in arbitrary characteristic. Let the characteristic of $A$ be $m$, now for any $n$ that is invertible in $A$ we would have $H^i_{fppf}(\mathrm{Spec}(A), \mathcal{I}_{fppf}) \cong H^i_{et}(\mathrm{Spec}(A), \mathcal{I}_{et})$ as a canonical generalization of \cite[Proposition 5.4, (2)]{VS2} in $fppf$-topology. Clearly, this isomorphism is compatible with the $[n]$-maps (see Theorem \ref{relkum}). Hence, comparing the corresponding Kummer sequences, the fppf relative Kummer sequence (Theorem \ref{relkum}) has the same cohomological consequences as the étale relative Kummer sequence of \cite[Proposition~6.2]{VS}. With this observation and adding example \ref{nonzero}, we conclude the following:

  \begin{thm}
   Let $f: A \to B$ be a finite subintegral extension of noetherian rings and $A$ be of characteristic $m > 0$. Then the following holds 
   \begin{enumerate}
   \item $H^i_{{fppf}}(\text{Spec}(A), \mu_{n}^{f}) = 0 $ for all $i \ge 0$ and any $n$ that is coprime to $m$.
   \item If $m = p$, a prime and $n=p^k$ for $k\geq1$, then the vanishing assertion of $(1)$ fails in general.
   \end{enumerate}
   \end{thm}

\bibliographystyle{alpha}
\bibliography{main}
\end{document}